\documentclass{amsart}
\usepackage{amssymb}
\usepackage{amsfonts}
\usepackage{amsmath}
\usepackage{psfig}
\usepackage{psfrag}
\usepackage{graphicx}

\newtheorem{theorem}{Theorem}[section]
\newtheorem{lemma}[theorem]{Lemma}
\newtheorem{proposition}[theorem]{Proposition}

\newtheorem{example}[theorem]{Example}
\newtheorem{corollary}[theorem]{Corollary}

\newtheorem{remark}[theorem]{Remark}

\numberwithin{equation}{section}

\begin{document}

\title{Justifications of spatial entropies of multi-dimensional symbolic dynamical systems}


\author{Wen-Guei Hu$^{\star}$}
\address{Department of Applied Mathematics, National Chiao Tung University, Hsinchu 300, Taiwan}
\email{wghu@mail.nctu.edu.tw}
\thanks{$^{\star\star}$ The first author would like to thank the ST Yau Center for partially supporting this research.}

\author{Song-Sun Lin$^{\dagger}$}
\address{Department of Applied Mathematics, National Chiao Tung University, Hsinchu 300, Taiwan}
\email{sslin@math.nctu.edu.tw}
\thanks{$^{\dagger}$ The second author would like to thank the National Science Council, R.O.C. (Contract No. NSC 103-2115-M-009-004- ) and the ST Yau Center for partially supporting this research.}

%

\begin{abstract}
The commonly used spatial entropy $h_{r}(\mathcal{U})$ of the multi-dimensional shift space $\mathcal{U}$ is the limit of growth rate of admissible local patterns on finite rectangular sublattices which expands to whole space $\mathbb{Z}^{d}$, $d\geq 2$. This work studies spatial entropy $h_{\Omega}(\mathcal{U})$ of shift space $\mathcal{U}$ on general expanding system $\Omega=\{\Omega(n)\}_{n=1}^{\infty}$ where $\Omega(n)$ is increasing finite sublattices and expands to $\mathbb{Z}^{d}$. $\Omega$ is called genuinely $d$-dimensional if $\Omega(n)$ contains no lower-dimensional part whose size is comparable to that of its $d$-dimensional part. We show that $h_{r}(\mathcal{U})$ is the supremum of $h_{\Omega}(\mathcal{U})$ for all
genuinely two-dimensional $\Omega$. Furthermore, when $\Omega$ is genuinely $d$-dimensional and satisfies certain conditions, then $h_{\Omega}(\mathcal{U})=h_{r}(\mathcal{U})$. On the contrary, when $\Omega(n)$ contains a lower-dimensional part, then $h_{r}(\mathcal{U})<h_{\Omega}(\mathcal{U})$ for some $\mathcal{U}$. Therefore, $h_{r}(\mathcal{U})$ is appropriate to be the $d$-dimensional spatial entropy.
\end{abstract}

\maketitle

\section{Introduction}

\hspace{0.5cm} Spatial entropy is known to measure the complexity of additive shift spaces and can be determined studying the growth rates of their admissible local patterns. Unlike in the one-dimensional case, subsequences of finite sublattices can approximate $\mathbb{Z}^{d}$, $d\geq 2$, in various ways. Among them, the rectangular sublattice is the most commonly used; see \cite{0-1,0-1-1,0-2,1,2,11,12,25,29,33-1,33-2,33-3}. This study investigates the spatial entropies of shift spaces according to their approximating ways to $\mathbb{Z}^{d}$, and compares them with commonly used rectangular spatial entropies.

For simplicity, this introduction considers only the case of $d=2$. Specifically, let $\mathcal{A}=\{0,1,\cdots,N-1\}$, $N\geq 2$, and
$\mathcal{U}\subseteq \mathcal{A}^{\mathbb{Z}^{2}}$ be an additive shift space, with $\mathbb{Z}^{2}$ as the two-dimensional lattice. Denote by $\Omega=\left\{ \Omega(n)\right\}_{n=1}^{\infty}$  an expanding system of finite lattice domains of $\mathbb{Z}^{2}$ with

\begin{equation}\label{eqn:1.0-1}
\Omega(n)\subset \Omega(n+1)
\end{equation}
and
\begin{equation}\label{eqn:1.0-2}
\underset{n=1}{\overset{\infty}{\bigcup}} \Omega(n)=\mathbb{Z}^{2}.
\end{equation}

Denote by $h_{\Omega}(\mathcal{U})$ the spatial entropy of $\mathcal{U}$ with respect to $\Omega$,

\begin{equation}\label{eqn:1.1}
h_{\Omega}(\mathcal{U})=\underset{n\rightarrow \infty}{\limsup}\hspace{0.1cm} \frac{1}{|\Omega(n)|}\log\Gamma(\Omega(n),\mathcal{U}),
\end{equation}
where $|\Omega(n)|$ is the cardinal number of $\Omega(n)$ and $\Gamma(\Omega(n),\mathcal{U})=\left|\mathcal{U}\mid_{\Omega(n)}\right|$, the cardinal number of $\mathcal{U}$ that is restricted on $\Omega(n)$. In particular, when $\Omega=\left\{\mathbb{Z}_{m\times n}\right\}_{m,n=1}^{\infty}$ is a sequence of rectangular sublattices, the rectangular entropy $h_{r}(\mathcal{U})$ is defined as

\begin{equation}\label{eqn:1.2}
h_{r}(\mathcal{U})=\underset{m,n\rightarrow \infty}{\limsup}\hspace{0.1cm} \frac{1}{mn}\log\Gamma_{m\times n}(\mathcal{U}),
\end{equation}
where $\Gamma_{m_{1}\times m_{2}}=\left|\mathcal{U}\mid_{\mathbb{Z}_{m_{1}\times m_{2}}}\right|$, $m_{1},m_{2}\geq 1$.

The sub-additive property of $\log\Gamma_{m_{1}\times m_{2}}$ in $m_{1}$ and $m_{2}$ is well known to imply that the limit of (\ref{eqn:1.2}) always exists and is commonly referred to as the spatial entropy in the literature \cite{12}. This study investigates $h_{\Omega}(\mathcal{U})$ for general $\Omega$ and $\mathcal{U}$ and its relationship with $h_{r}(\mathcal{U})$.

This study is directly motivated by our recent study of the spatial entropy of a multiplicative integer system \cite{0-1}. Multiplicative integer systems arise in the study of multiple ergodic averages and have been intensively studied in recent years; see \cite{17,18,26,32,33,43,44,45} and the references therein. One of the important issue is to compute Minkowski (box) dimension and Hausdorff dimension of such systems and to compare them. Unlike additive shift spaces, these two dimensions are unequal for most known multiplicative integer systems; see Fan \emph{et al.} \cite{17,18}, Kenyon \emph{et al.} \cite{32,33} and Peres \emph{et al.} \cite{43}. In \cite{0-1}, we introduce the spatial entropy to compute the Minkowski dimension. It is briefly introduced it as follows.

A multiplicative integer system $\mathbb{X}\subset \{0,1,2,\cdots,N-1\}^{\mathbb{N}}$ satisfies $\left(x_{rk}\right)\in\mathbb{X}$ for any $\left(x_{k}\right)\in\mathbb{X}$ and $r\geq 1$, where $\mathbb{N}$ is the set of all natural numbers. The spatial entropy $h(\mathbb{X})$ of $\mathbb{X}$ is defined by

\begin{equation}\label{eqn:1.3}
h_{r}(\mathbb{X})=\underset{n\rightarrow \infty}{\limsup}\hspace{0.1cm} \frac{1}{n}\log\left|X_{n}\right|,
\end{equation}
where $X_{n}=\mathbb{X}\mid_{\mathbb{Z}_{n}}$ and $\mathbb{Z}_{n}=\{1,2,\cdots,n\}$. For any $q\geq2$, denote by the multiplicative system

\begin{equation}\label{eqn:1.4}
\mathbb{X}_{q}^{0}=\left\{ (x_{k})\in\{0,1\}^{\mathbb{N}} \hspace{0.1cm} \mid \hspace{0.1cm} x_{k}x_{qk}=0, k\geq 1\right\}.
\end{equation}

In \cite{0-1}, it is verified that

\begin{equation}\label{eqn:1.5}
h(\mathbb{X}_{q}^{0})=(q-1)^{2} \hspace{0.1cm} \underset{k=1}{\overset{\infty}{\sum}}\hspace{0.1cm} \frac{1}{q^{k+1}}\log a_{k},
\end{equation}
where $a_{k}=a_{k-1}+a_{k-2}$, $k\geq 3$, is the Fibonacci number with $a_{1}=2$ and $a_{2}=3$. To obtain (\ref{eqn:1.5}), $\mathbb{N}$ is rearranged as the first quadrant of a two-dimensional lattice as

\begin{equation}\label{eqn:1.6}
\mathbb{N}=\mathcal{I}_{q}\times M_{q},
\end{equation}
where $M_{q}=\left\{  q^{k}  \hspace{0.1cm} \mid \hspace{0.1cm} q\geq 0   \right\}$ and $\mathcal{I}_{q}=\left\{ i\in\mathbb{N} \hspace{0.1cm} \mid \hspace{0.1cm} q\nmid i  \right\}$; see Fig. 1.1 for $q=2$.

\begin{equation*}
\begin{array}{c}
\psfrag{y}{$\mathbb{M}_{2}$}
\psfrag{z}{$\mathcal{I}_{2}$}
\includegraphics[scale=0.9]{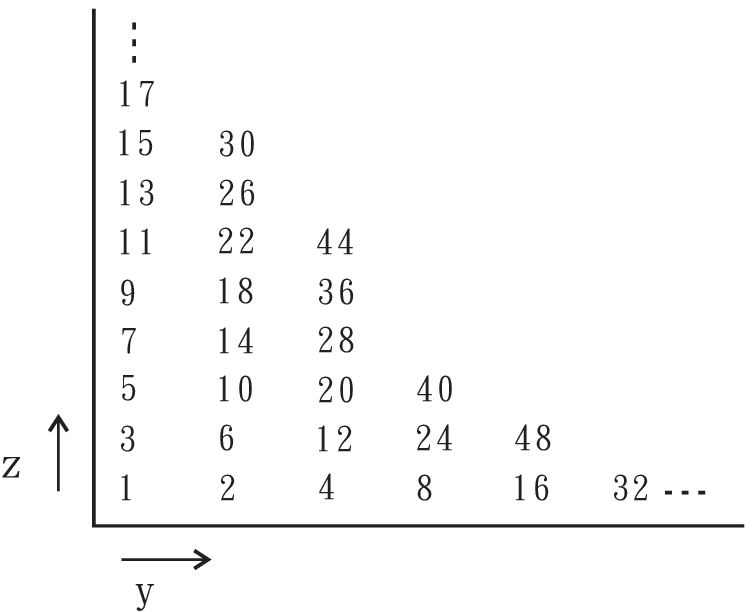}
\end{array}
\end{equation*}
\begin{equation*}
\text{Figure 1.1.}
\end{equation*}

The entropy $h(\mathbb{X}_{q}^{0})$ can be easily obtained using the following formula

\begin{equation}\label{eqn:1.7}
q^{n}= (n+1)+n(q-2)+(q-1)^{2}\hspace{0.1cm} \underset{k=1}{\overset{n-1}{\sum}}\hspace{0.1cm} k q^{n-n-k}.
\end{equation}
After a second thought, the result thus obtained can also be interpreted as a study of the two-dimensional entropy of the additive shift of finite type $\mathcal{U}_{\mathcal{B}}\subset \{0,1\}^{\mathbb{Z}^{2}}$, where the forbidden set $\mathcal{F}$ of $\mathcal{U}_{\mathcal{B}}$ is $\mathcal{F}=\left\{\begin{array}{c} \includegraphics[scale=0.6]{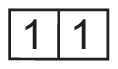} \end{array}\right\}$, meaning that the basic set of admissible patterns $\mathcal{B}\subset \{0,1\}^{\mathbb{Z}_{2\times 2}}$ is given as

\begin{equation}\label{eqn:1.8}
\mathcal{B}=\left\{\begin{array}{ccccccccc}
 \includegraphics[scale=0.6]{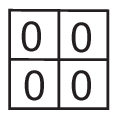}, &
  \includegraphics[scale=0.6]{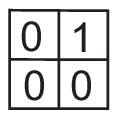}, &
   \includegraphics[scale=0.6]{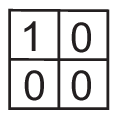}, &
    \includegraphics[scale=0.6]{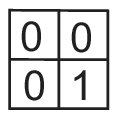}, &
     \includegraphics[scale=0.6]{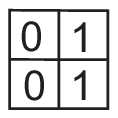}, &
      \includegraphics[scale=0.6]{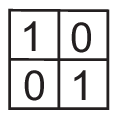}, &
       \includegraphics[scale=0.6]{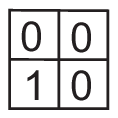}, &
        \includegraphics[scale=0.6]{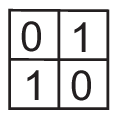}, &
         \includegraphics[scale=0.6]{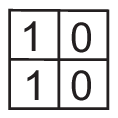}
\end{array}\right\}.
\end{equation}

$\mathcal{U}_{\mathcal{B}}$ is considered to satisfy the Golden-Mean condition $x_{i,j}x_{i+1,j}=0$ in the horizontal direction and to be unconstrained in the vertical direction. Now, denote by

\begin{equation}\label{eqn:1.9}
\Omega_{q}^{+}(n)=\left\{ k\in \mathbb{N}  \hspace{0.1cm} \mid \hspace{0.1cm} 1\leq k\leq q^{n}  \right\}
\end{equation}
in $\mathcal{I}_{q}\times M_{q}$, as presented in Fig. 1.2 (a). By reflecting $\Omega_{q}^{+}(n)$ in the horizontal and vertical axes, the lattice $\Omega_{q}(n)\subset \mathbb{Z}^{2}$ is constructed as shown in Fig. 1.2 (b).

\begin{equation*}
\psfrag{a}{{\footnotesize $q^{n}-1$}}
\psfrag{b}{{\footnotesize $q^{n}$}}
\psfrag{n}{{\footnotesize $n$}}
\psfrag{1}{$1$}
\psfrag{c}{{\footnotesize $M_{q}$}}
\psfrag{d}{{\footnotesize $\mathcal{I}_{q}$}}
\psfrag{e}{(a) $\Omega_{q}^{+}(n)$}
\psfrag{f}{(b) $\Omega_{q}(n)$}
\psfrag{g}{(c) }
\psfrag{p}{{\footnotesize $\omega(n)$}}
 \includegraphics[scale=0.8]{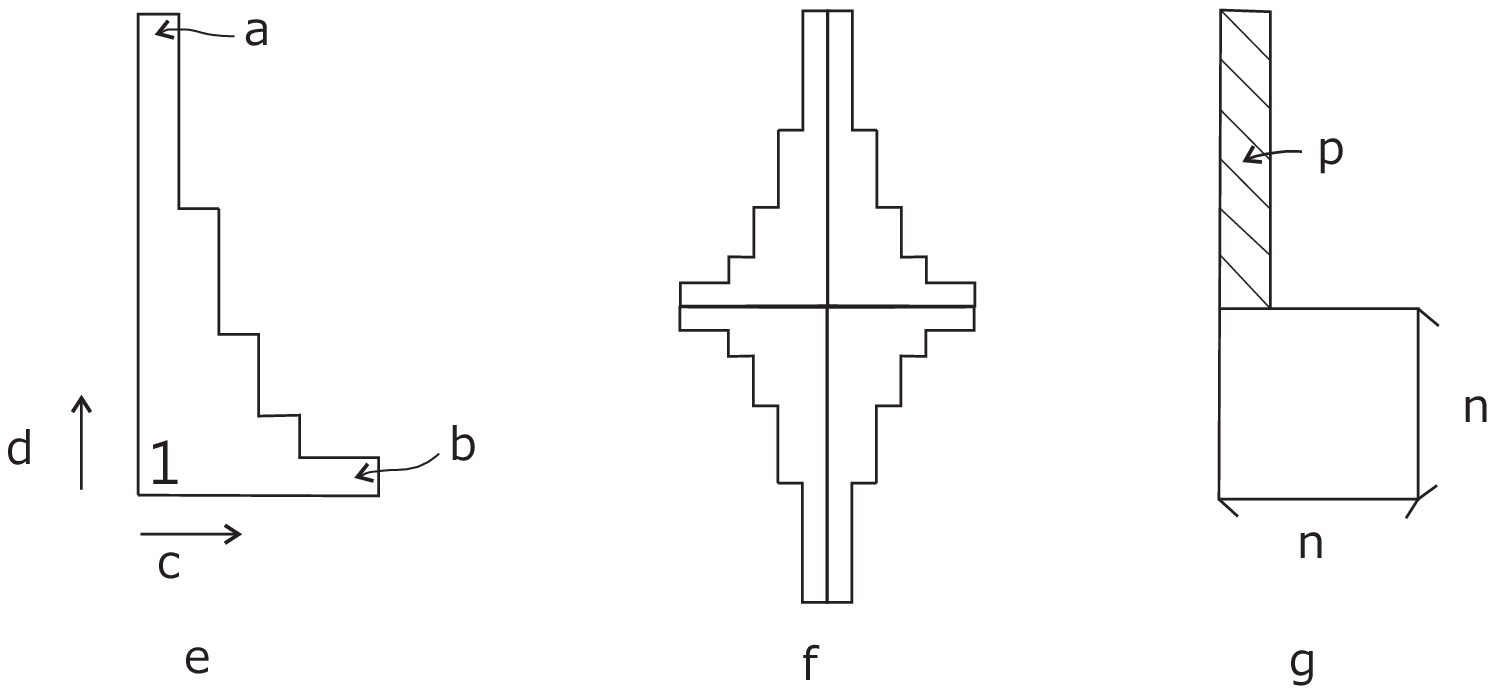}
\end{equation*}
\begin{equation*}
\text{Figure 1.2.}
\end{equation*}

Then, for each $q\geq 2$, $\Omega_{q}=\{\Omega_{q}(n)\}_{n=1}^{\infty}$ gives an approximation to $\mathbb{Z}^{2}$. Clearly, (\ref{eqn:1.7}) implies

\begin{equation}\label{eqn:1.10}
\Gamma\left(\Omega_{q}(n),\mathcal{U}_{\mathcal{B}}\right)= a_{2(n+1)}^{2}a_{2n}^{2(q-2)} \hspace{0.1cm}\left( \underset{k=1}{\overset{n-1}{\prod}} \hspace{0.1cm} a_{2k}^{2(q-1)^{2}q^{n-n-k}}\right)
\end{equation}
and $\left| \Omega_{q} (n)\right|= 4\left| \Omega_{q}^{+} (n)\right|=4q^{n}$. Therefore,

\begin{equation}\label{eqn:1.11}
h_{\Omega_{q}}(\mathcal{U}_{\mathcal{B}})= \frac{1}{2} (q-1)^{2}  \hspace{0.1cm}\underset{k=1}{\overset{n-1}{\sum}} \hspace{0.1cm} \frac{1}{q^{k+1}}\log a_{2k},
\end{equation}
where $\Omega_{q}=\left\{ \Omega_{q}(n)\right\}_{n=1}^{\infty}$.

It can be easily verified that

\begin{equation}\label{eqn:1.12}
h_{r}(\mathcal{U}_{\mathcal{B}})= \log g.
\end{equation}
$h_{\Omega_{q}}(\mathcal{U}_{\mathcal{B}})$ can be shown to be strictly increasing in $q$ and tends to $\log 2$ as $q\rightarrow\infty$, and

\begin{equation}\label{eqn:1.13}
h_{\Omega_{q}}(\mathcal{U}_{\mathcal{B}})> h_{r}(\mathcal{U}_{\mathcal{B}})
\end{equation}
for any $q\geq 2$. Therefore, shift space $\mathcal{U}_{\mathcal{B}}$ has infinitely many spatial entropies $h_{\Omega_{q}}(\mathcal{U}_{\mathcal{B}})$.

The mechanism of (\ref{eqn:1.13}) can be explained as follows. From (\ref{eqn:1.7}), $\Omega_{q}^{+}(n)$ includes $(q-1)^{2}q_{n-1-k}$ copies of $\mathbb{Z}_{k}$, so implying there are two copies of $\mathbb{Z}_{2k\times (q-1)^{2}q_{n-1-k}}$ in $\Omega_{q}(n)$ for $1\leq k\leq n-1$. For fixed $k$ and large $n$, $\mathbb{Z}_{2k\times (q-1)^{2}q_{n-1-k}}$ has the form of a long stick in the vertical direction, as a one-dimensional object whose size is comparable to that of its two-dimensional part. Hence, lack of a constraint in the vertical direction of $\mathcal{U}_{\mathcal{B}}$ in $\Omega_{q}$ provides more admissible patterns than in two-dimensional rectangular lattice $\mathbb{Z}_{m_{1}\times n_{1}}$ with $m_{1}n_{1}=4q^{n}$.

The result (\ref{eqn:1.13}) indicates that the spatial entropy $h_{\Omega}(\mathcal{U})$ of (\ref{eqn:1.1}) that describes the growth rate of patterns of $\mathcal{U}$ depends
very strongly on how $\Omega=\{\Omega(n)\}$ approximates $\mathbb{Z}^{2}$. Therefore, the fundamental problem of when $h_{\Omega}(\mathcal{U})=h_{r}(\mathcal{U})$ and when $h_{\Omega}(\mathcal{U})\neq h_{r}(\mathcal{U})$ must be investigated.

The only known relevant result in the literature \cite{0} is that of Ballister \emph{et al.} who proved $h_{\Omega}(\mathcal{U})=h_{r}(\mathcal{U})$ when $\Omega=\{\Omega(n)\}_{n=1}^{\infty}$ is a sequence of bounded convex sets whose inradii tend to infinity for any additive shift space.

 The main results obtained herein for $h_{\Omega}(\mathcal{U})=h_{r}(\mathcal{U})$ are as follows. Let $\mathbb{Z}_{m_{1}\times n_{1}}$ be the smallest rectangular lattice that contains $\Omega(n)$, and let $\partial \Omega (n)$ be the boundary of $\Omega(n)$. Denote by $\Omega '(n)$ the complement of $\Omega(n)$ in $\mathbb{Z}_{m_{1}\times n_{1}}$. Now, the following result holds for $h_{\Omega}(\mathcal{U})=h_{r}(\mathcal{U})$.

\begin{theorem}
\label{Theorem:1.1}
Let $\mathcal{U}\subseteq\{0,1,\cdots,N-1\}^{\mathbb{Z}^{2}}$ be an additive shift space. If
\begin{equation}\label{eqn:1.14}
\underset{n\rightarrow\infty}{\limsup} \frac{|\partial\Omega(n)|}{|\Omega(n)|}=0,
\end{equation}
then
\begin{equation}\label{eqn:1.14-1}
h_{\Omega}(\mathcal{U})\leq h_{r}(\mathcal{U}).
\end{equation}
Furthermore, if (\ref{eqn:1.14}) and
\begin{equation}\label{eqn:1.15}
\sup\left\{\frac{|\Omega'(n)|}{|\Omega(n)|}: n\geq 1\right\}<\infty
\end{equation}
hold, then
\begin{equation}\label{eqn:1.16}
h_{\Omega}(\mathcal{U})= h_{r}(\mathcal{U}).
\end{equation}
\end{theorem}

Notably, (\ref{eqn:1.14}) and (\ref{eqn:1.15}) are geometrical conditions that apply for all shift spaces. Roughly, $\Omega(n)$ contains no lower-dimensional part whose size is comparable to that of its two-dimensional part. No condition on the shape of $\Omega$, unlike in the work of Ballister \emph{et al.} \cite{0}, is required. Whether (\ref{eqn:1.14}) alone can imply (\ref{eqn:1.16}), such that (\ref{eqn:1.15}) is unnecessary, is of interest. If $\mathcal{U}$ satisfies a certain mixing condition, then (\ref{eqn:1.14}) alone implies (\ref{eqn:1.16}). For example, when $\mathcal{U}$ is block gluing \cite{0-2,11}, a favorable result is obtained.

\begin{theorem}
\label{Theorem:1.2}
If $\mathcal{U}$ is a block gluing shift space and $\Omega$ satisfies (\ref{eqn:1.14}), then (\ref{eqn:1.16}) holds.
\end{theorem}

From Theorems \ref{Theorem:1.1} and \ref{Theorem:1.2}, some shift spaces will satisfy $h_{\Omega}(\mathcal{U})>h_{r}(\mathcal{U})$ only if (\ref{eqn:1.14}) is violated like $\Omega_{q}$, $q\geq 2$. Roughly, $\Omega(n)$ must contain a lower-dimensional part $\omega(n)$ whose size is comparable to that of its two-dimensional part, meaning that $|\omega(n)|/\left|\Omega(n)\setminus \omega(n)\right|$ is non-zero as $n\rightarrow\infty$. See Fig. 1.2 (c) with $\underset{n\rightarrow\infty}{\lim}\frac{|\omega(n)|}{n^{2}}>0$.

Some notations must be introduced before the results herein can be presented. Given a finite lattice $\mathbb{L}\subset\mathbb{Z}^{2}$, for $m\geq 1$, a point $(i,j)\in \mathbb{L}$ has horizontal length $m$ in $\mathbb{L}$ if
$m$ is the largest positive integer such that there exists a $m\times 1$ rectangular lattice in $\mathbb{L}$ that contains $(i,j)$.
 %

Let $\Omega=\left\{\Omega(n)\right\}_{n=1}^{\infty}$. 
%
For $m\geq 1$, define the subset $\Omega_{m}^{(h)}(n)$ of $\omega(n)$ with horizontal length $m$ by

\begin{equation}\label{eqn:4.1}
\Omega_{m}^{(h)}(n)=\left\{(i,j) \in\Omega(n) \hspace{0.1cm}\mid \hspace{0.1cm}  (i,j) \text{ has horizontal length }m \text{ in }\Omega(n) \right\}.
\end{equation}
%
Denote by $\beta_{m}^{(h)}(n)=\left|\Omega_{m}^{(h)}(n) \right|$. The subset $\Omega_{m}^{(v)}(n)$ of $\omega(n)$ with vertical length $m$ can be similarly defined. Denote by $\beta_{m}^{(v)}(n)=\left|\Omega_{m}^{(v)}(n) \right|$.

\begin{theorem}
\label{Theorem:1.3}
If there exists $m\geq 1$ such that

\begin{equation}\label{eqn:1.17}
\underset{n\rightarrow\infty}{\limsup} \frac{\beta_{m}^{(h)}(n)}{|\Omega(n)|}>0
\end{equation}
or
\begin{equation}\label{eqn:1.17-1}
\underset{n\rightarrow\infty}{\limsup} \frac{\beta_{m}^{(v)}(n)}{|\Omega(n)|}>0,
\end{equation}
then there exists an additive shift of finite type $\mathcal{U}$ such that
\begin{equation}\label{eqn:1.18}
h_{\Omega}(\mathcal{U})>h_{r}(\mathcal{U}).
\end{equation}

\end{theorem}

Notably, condition (\ref{eqn:1.17}) or (\ref{eqn:1.17-1}) implies that $\Omega(n)$ contains a lower-dimensional part whose size is comparable to that of its two-dimensional part.

Theorem \ref{Theorem:1.3} can be extended when $\mathcal{U}$ satisfies a certain mixing condition and $\Omega$ contains some non-negligible lower-dimensional parts, as follows. The growth rate of the lower-dimensional parts must be demonstrated. Johnson \emph{et al.} \cite{30} previously introduced the projectional entropy $h_{\mathbb{L}}(\mathcal{U})$ of a $d$-dimensional shift space $\mathcal{U}$, $d\geq2$, where $\mathbb{L}$ is an $r$-dimensional sublattice of $\mathbb{Z}^{d}$, $1\leq r <d$. Moreover, those authors proved

\begin{equation}\label{eqn:1.20-0}
h_{\mathbb{L}}(\mathcal{U})\geq h_{top}(\mathcal{U}),
\end{equation}
where $h_{top}(\mathcal{U})$ is the topological entropy of $\mathcal{U}$. Notably,

\begin{equation}\label{eqn:1.20-1}
h_{top}(\mathcal{U})=h_{r}(\mathcal{U}).
\end{equation}

For any $\mathbb{Z}^{2}$ shift space $\mathcal{U}$, let $\hat{h}^{(1)}(\mathcal{U})$ be the supremum of projectional entropy for all one-dimensional sublattices:

\begin{equation}\label{eqn:1.19-4-1}
\hat{h}^{(1)}(\mathcal{U})=\sup \left\{h_{\mathbb{L}}(\mathcal{U})  \hspace{0.1cm} \mid \hspace{0.1cm} \mathbb{L} \text{ is an one-dimensional sublattice} \right\}.
\end{equation}
Clearly,

\begin{equation}\label{eqn:1.20-2}
\hat{h}^{(1)}(\mathcal{U})\geq h_{r}(\mathcal{U}).
\end{equation}
Then, the following can be established.

\begin{theorem}
\label{Theorem:1.4}
Let $\mathcal{U}$ be a block gluing $\mathbb{Z}^{2}$ shift space. If

\begin{equation}\label{eqn:1.20-3}
 h_{r}(\mathcal{U}) < \hat{h}^{(1)}(\mathcal{U}),
\end{equation}
then for any $h\in\left[h_{r}(\mathcal{U}),\hat{h}^{(1)}(\mathcal{U})\right)$, there exists $\Omega=\{\Omega(n)\}_{n=1}^{\infty}$ such that

\begin{equation}\label{eqn:1.20-4}
h_{\Omega}(\mathcal{U})=h.
\end{equation}
Furthermore, if $\hat{h}^{(1)}$ can be attained by some one-dimensional sublattice $\mathbb{L}'$, then there exists $\Omega=\{\Omega(n)\}_{n=1}^{\infty}$ such that
\begin{equation}\label{eqn:1.20-5}
h_{\Omega}(\mathcal{U})=\hat{h}^{(1)}=h_{\mathbb{L}'}(\mathcal{U}).
\end{equation}
\end{theorem}

Condition (\ref{eqn:1.20-3}) has been discussed elsewhere \cite{30} as described in Section 4.

 With respect to Theorems \ref{Theorem:1.1}$\sim$\ref{Theorem:1.4}, $\Omega=\{\Omega(n)\}_{n=1}^{\infty}$ is called a genuinely two-dimensional approximation to $\mathbb{Z}^{2}$ if $\Omega$ satisfies (\ref{eqn:1.0-1}), (\ref{eqn:1.0-2}) and (\ref{eqn:1.14}). Therefore, the genuinely two-dimensional spatial entropy can be defined as follows.

\begin{equation}\label{eqn:1.20}
h(\mathcal{U})=\sup\left\{h_{\Omega}(\mathcal{U}) \hspace{0.1cm} \mid\hspace{0.1cm} \Omega \text{ satisfies (\ref{eqn:1.0-1}), (\ref{eqn:1.0-2}) and (\ref{eqn:1.14}) }   \right\},
\end{equation}
so $h(\mathcal{U})$ measures the maximum growth rate of admissible local patterns for all genuinely two-dimensional approximations to $\mathbb{Z}^{2}$. Clearly,  Theorem \ref{Theorem:1.1} implies

\begin{equation}\label{eqn:1.21}
h_{r}(\mathcal{U})=h(\mathcal{U}).
\end{equation}
Therefore, $h_{r}(\mathcal{U})$ can be appropriately said to be the two-dimensional spatial entropy.

No shift space $\mathcal{U}$ and $\Omega$ that satisfy

\begin{equation}\label{eqn:1.19-5}
h_{\Omega}(\mathcal{U})<h_{r}(\mathcal{U})
\end{equation}
has yet been found. The problem of whether or not the rectangular spatial entropy is the minimum entropy, i.e., whether or not

\begin{equation}\label{eqn:1.22}
h_{r}(\mathcal{U})=\inf \left\{ h_{\Omega}(\mathcal{U}) \hspace{0.1cm} \mid\hspace{0.1cm} \Omega \text{ satisfies (\ref{eqn:1.0-1}) and (\ref{eqn:1.0-2})  } \right\},
\end{equation}
needs further investigation.

The two-dimensional results can be generalized to higher-dimensional cases on $\mathbb{Z}^{d}$, $d\geq 3$. The derails are omitted for brevity.

The rest of this paper is arranged as follows. Section 2 introduces some useful notation and considers $h_{r}(\mathcal{U})$ in more detail. Section 3 proves a general version of Theorems
\ref{Theorem:1.1} and \ref{Theorem:1.2}. Section 4 proves Theorems \ref{Theorem:1.3} and  \ref{Theorem:1.4}.

\section{Rectangular entropy}

\hspace{0.5cm} This section introduces notation that will be useful in considering $h_{\Omega}(\mathcal{U})=h_{r}(\mathcal{U})$ in Section 3, and it presents some properties of rectangular entropy.

Firstly, the notation is introduced. Let $\mathcal{A}=\{0,1,\cdots,N-1\}$, $N\geq 2$. For any vector $\mathbf{v}\in \mathbb{Z}^{2}$, the shift map $\sigma^{\mathbf{v}}: \mathcal{A}^{\mathbb{Z}^{2}}\rightarrow \mathcal{A}^{\mathbb{Z}^{2}}$ is defined by

\begin{equation*}
\left(\sigma^{\mathbf{v}}(x)\right)_{i,j}=x_{(i,j)+\mathbf{v}}
\end{equation*}
for all $(i,j)\in \mathbb{Z}^{2}$. An additive shift space $\mathcal{U}\subseteq \mathcal{A}^{\mathbb{Z}^{2}}$ is transitive variant and closed. Equivalently, an additive shift space can be defined by a forbidden set as follows.
Let forbidden set

\begin{equation*}
\mathcal{F}=\underset{\mathbb{S}\in\mathcal{S}}{\bigcup} \hspace{0.15cm} \mathcal{A}^{\mathbb{L}},
\end{equation*}
where $\mathcal{S}$ is the set of shapes $\mathbb{\mathbb{S}}\subset\mathbb{Z}^{2}$, $|\mathbb{S}|< \infty$. Then, the additive shift space $\mathcal{U}=\mathcal{U}_{\mathcal{F}}$ of $\mathcal{F}$ is given by

\begin{equation*}
 \mathcal{U}_{\mathcal{F}}=\left\{x\in \mathcal{A}^{\mathbb{Z}^{2}} \hspace{0.1cm} \mid \hspace{0.1cm} \sigma^{\mathbf{v}}(x)\mid_{\mathbb{S}}\notin \mathcal{F} \text{ for all }\mathbb{L}\in\mathcal{S} \text{ and }\mathbf{v\in\mathbb{Z}^{2}} \right\}.
\end{equation*}
$\mathcal{U}_{\mathcal{F}}$ is called a shift of finite type if $|\mathcal{F}|< \infty$.

Given an additive shift $\mathcal{U}$, for $\mathbb{L}\subseteq \mathbb{Z}^{2}$, the set of admissible patterns on $\mathbb{L}$ is defined by
\begin{equation*}
\Sigma_{\mathbb{L}}(\mathcal{U})=\mathcal{U}\mid_{\mathbb{L}}=\left\{x\mid_{\mathbb{L}}\hspace{0.1cm} \mid \hspace{0.1cm} x\in\mathcal{U} \right\}.
\end{equation*}
Denote by $\Gamma(\mathbb{L},\mathcal{U})=\left| \Sigma_{\mathbb{L}}(\mathcal{U}) \right|$  the cardinal number of $\Sigma_{\mathbb{L}}(\mathcal{U})$.

For any $m,n\geq 1$ and $(i,j)\in\mathbb{Z}^{2}$, the $m\times n$ rectangular lattice with the left-bottom vertex $(i,j)$ is denoted by

\begin{equation*}
\mathbb{Z}_{m\times n}((i,j))=\left\{(i+n_{1},j+n_{2})\mid 0\leq n_{1}\leq m-1,0\leq n_{2}\leq n-1 \right\}.
\end{equation*}
In particular,

\begin{equation*}
\mathbb{Z}_{m\times n}=\mathbb{Z}_{m\times n}((0,0)).
\end{equation*}
Then, let

\begin{equation*}
\Sigma_{m\times n}(\mathcal{U}) = \Sigma_{\mathbb{Z}_{m\times n}}(\mathcal{U})
\end{equation*}
and

\begin{equation*}
\Gamma_{m\times n}(\mathcal{U}) = \Gamma_{\mathbb{Z}_{m\times n}}(\mathcal{U}).
\end{equation*}
For a additive shift $\mathcal{U}$, it is well-known that
\begin{equation*}
  h_{r}(\mathcal{U})= \underset{m,n\geq 1}{\inf} \frac{1}{mn}\log \Gamma_{m\times n}(\mathcal{U})
\end{equation*}
by the sub-additive property of $\log \Gamma_{m\times n}(\mathcal{U})$; see \cite{12}.

A finite subset $\mathbb{T}\subset\mathbb{Z}^{2}$ is called a tessellation of $\mathbb{Z}^{2}$ if there exists a sequence
$\mathbf{v}_{1},\mathbf{v}_{2},\cdots$ in $\mathbb{Z}^{2}$ such that $\underset{i=1}{\overset{\infty}{\bigcup}}\hspace{0.1cm}\mathbb{T}+\mathbf{v}_{i}$ is a partition of $\mathbb{Z}^{2}$. For example, $\mathbb{T}$ can be a rectangle, a parallelogram, or an L-shaped lattice; see Fig. 2.1.

\begin{equation*}
 \includegraphics[scale=0.4]{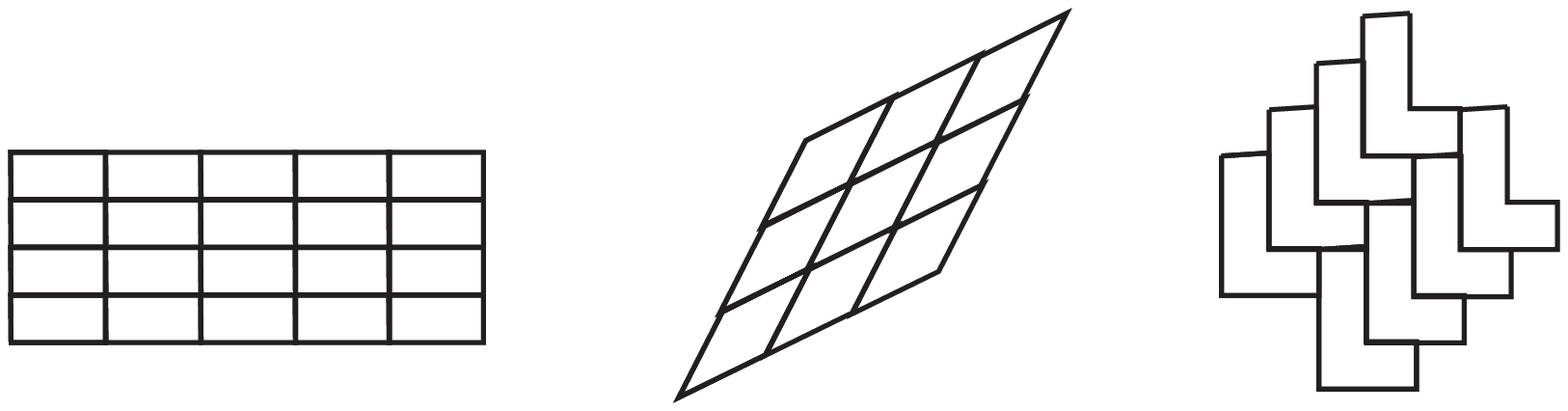}
\end{equation*}
\begin{equation*}
\text{Figure 2.1.}
\end{equation*}
Ballister \emph{et al.} \cite{0} proved

\begin{equation}\label{eqn:2.0-1}
\frac{1}{|\mathbb{T}|}\log \Gamma_{\mathbb{T}}(\mathcal{U})\geq h_{r}(\mathcal{U})
\end{equation}
for any additive shift $\mathcal{U}$.

Now, some mixing properties and notation are introduced for use in the examination of $h_{\Omega}(\mathcal{U})=h_{r}(\mathcal{U})$ in Section 3. Let $d$ be the Euclidean metric in $\mathbb{Z}^{2}$.
A shift space $\mathcal{U}$ is called block gluing if a number $M(\Sigma)\geq 1$ exists such that for any two allowable patterns $U_{1}\in \mathcal{U}\mid_{\mathbb{R}_{1}}$ and $U_{2}\in\mathcal{U}\mid_{\mathbb{R}_{2}}$ with $d(\mathbb{R}_{1},\mathbb{R}_{2})\geq M$, where $\mathbb{R}_{1}=\mathbb{Z}_{m_{1}\times n_{1}}((i_{1},j_{1}))$ and  $\mathbb{R}_{2}=\mathbb{Z}_{m_{2}\times n_{2}}((i_{2},j_{2}))$, $m_{l},n_{l}\geq1 $ and $(i_{l},j_{l})\in\mathbb{Z}^{2}$, $l\in\{1,2\}$, there exists a global pattern $W\in\mathcal{U}$ with $W\mid_{\mathbb{R}_{1}}=U_{1}$ and $W\mid_{\mathbb{R}_{2}}=U_{2}$; see \cite{11}.

In particular,
a shift space $\mathcal{U}$ is called horizontally block gluing if only $\mathbb{R}_{1}=\mathbb{Z}_{m_{1}\times n_{1}}((i_{1},j))$ and  $\mathbb{R}_{2}=\mathbb{Z}_{m_{2}\times n_{2}}((i_{2},j))$ is considered. Similarly,
$\mathcal{U}$ is called vertically block gluing if only $\mathbb{R}_{1}=\mathbb{Z}_{m_{1}\times n_{1}}((i,j_{1}))$ and  $\mathbb{R}_{2}=\mathbb{Z}_{m_{2}\times n_{2}}((i,j_{2}))$ is considered.

A sequence of finite lattices $\Omega=\left\{ \Omega(n)\right\}_{n=1}^{\infty}$ is called horizontally (or vertically) decomposable if there exists $m\geq 1$ such that $\Omega(n)$ can be decomposed into $m$ disjoint rectangular lattices by cutting $\Omega(n)$ along horizontal (or vertical) lines for all $n\geq 1$.

\begin{example}
\label{example:1.1}
Consider $\Omega=\left\{ \Omega(n)\right\}_{n=2}^{\infty}$, where $\Omega(n)$ is described in Fig. 2.2. Clearly, $\Omega$ is horizontally decomposable.
\begin{equation*}
\begin{array}{ccccccccc}
\psfrag{b}{{\footnotesize$n^{2}$}}
\psfrag{a}{$n$}
\psfrag{c}{{\footnotesize cutting line}}
\includegraphics[scale=1.0]{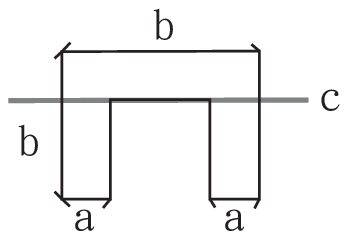}
&
&
&
& & & &
\psfrag{b}{{\footnotesize$n^{2}$}}
\psfrag{a}{$n$}
\psfrag{c}{{\footnotesize cutting line} }
\hspace{1.0cm}\includegraphics[scale=1.0]{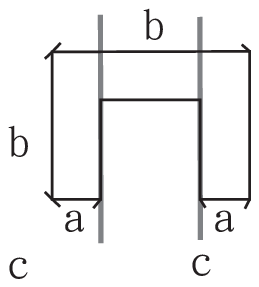}
\end{array}
\end{equation*}
\begin{equation*}
\text{Figure 2.2.}
\end{equation*}

\end{example}

The following proposition shows that the rectangular entropy is always strictly smaller than $\frac{1}{mn}\log \Gamma_{m\times n}(\mathcal{U})$, $m,n\geq 1$, except for full shifts.

\begin{proposition}
\label{proposition:1.1}
For any nonempty additive shift space $\mathcal{U}$. If

\begin{equation}\label{eqn:2.1}
h_{r}(\mathcal{U})=\frac{1}{mn}\log \Gamma_{m\times n}(\mathcal{U})
\end{equation}
for some $m,n\geq 1$, then $\mathcal{U}$ is a full shift. In particular, if $\mathcal{U}$ is not a full shift, then

\begin{equation}\label{eqn:2.1-1}
h_{r}(\mathcal{U})<\frac{1}{mn}\log \Gamma_{m\times n}(\mathcal{U})
\end{equation}
for any $m,n\geq 1$.

\end{proposition}

\begin{proof}
Let $\mathcal{U}=\mathcal{U}_{F}\subseteq \mathcal{A}^{\mathbb{Z}^{2}}$ for some $\mathcal{F}=\underset{\mathbb{S}\in\mathcal{S}}{\bigcup} \hspace{0.15cm} \mathcal{A}^{\mathbb{S}}$. Without loss of generality, $\mathcal{F}$ is assumed to be simplified, that is, $\mathcal{F}$ satisfies the following two conditions.
\begin{enumerate}

\item[(i)] $\mathcal{A}=\{0,1,\cdots,N-1\}$, for some $N\geq 2$, is the smallest alphabet for $\mathcal{U}$,

\item[(ii)] if $U$ is a forbidden pattern on $\mathbb{S}\in\mathcal{S}$, then for any $\mathbb{S}'\subsetneq\mathbb{S}$, there exists $x\in \mathcal{U}$ such that
$x\mid_{\mathbb{S}'}=U\mid_{\mathbb{S}'}$.

\end{enumerate}

By the subadditive property, we have

\begin{equation*}
 \Gamma_{\alpha m\times \beta n}(\mathcal{U})\leq \Gamma_{m\times n}^{\alpha\beta}(\mathcal{U})
\end{equation*}
for all $\alpha,\beta \geq 1$. Then, from (\ref{eqn:2.1}), it can be verified that

\begin{equation}\label{eqn:2.2}
 \Gamma_{\alpha m\times \beta n}(\mathcal{U})= \Gamma_{m\times n}^{\alpha\beta}(\mathcal{U})
\end{equation}
for all  $\alpha,\beta \geq 1$. This means that for any two patterns $U_{1}$ and $U_{2}$ in $\Sigma_{k m\times l n}(\mathcal{U})$, $k,l\geq 1$, they can be tessellated together in horizontal or vertical direction to be a pattern in $\Sigma_{2 k m\times l n}(\mathcal{U})$ or $\Sigma_{k m\times 2 l n}(\mathcal{U})$, respectively.

We will prove $\mathcal{F}=\emptyset$ by contradiction. Suppose there exists a forbidden pattern $U\in\mathcal{F}$ on $\mathbb{S}$ and $\mathbb{S}\subseteq \mathbb{Z}_{a m \times bn}$ for some $a,b\geq 1$. From condition (i), $\mathbb{S}\neq \mathbb{Z}_{1\times 1}$. Then, by cutting $\mathbb{S}$ along a horizontal or vertical line, $\mathbb{S}$ can be decomposed as $\mathbb{S}=\mathbb{S}_{1}\cup \mathbb{S}_{2}$ where $\mathbb{S}_{1}$ and $\mathbb{S}_{2}$ are not empty and $\mathbb{S}_{1}\cap \mathbb{S}_{2}=\emptyset$.
Here, only the case by cutting along a horizontal line such that $\mathbb{S}_{1}$ is on the left of $\mathbb{S}_{2}$ is considered. Similarly, the other cases can be proven. Then, there exists $(i,j)\in\mathbb{Z}^{2}$ such that $\mathbb{S}_{1}\subseteq\mathbb{Z}_{a m \times bn}((i,j))$ and $\mathbb{S}_{2}\subseteq\mathbb{Z}_{a m \times bn}((i+am,j+bn))$; see Fig. 2.3.

\begin{equation*}
\psfrag{a}{$\mathbb{S}_{1}$ }
\psfrag{b}{$\mathbb{S}_{2}$ }
\psfrag{c}{{\scriptsize$\mathbb{Z}_{a m \times bn}((i,j))$ }}
\psfrag{d}{{\scriptsize$\mathbb{Z}_{a m \times bn}((i+am,j+bn)) $}}
\psfrag{e}{{\small cutting line}}
\includegraphics[scale=1.2]{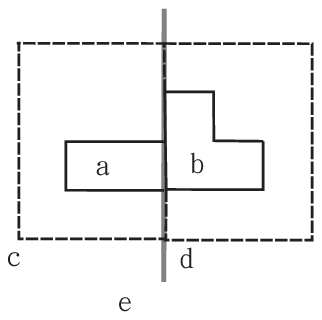}
\end{equation*}
\begin{equation*}
\text{Figure 2.3.}
\end{equation*}

Condition (ii) implies that there exist $x_{1},x_{2}\in\mathcal{U}$ such that $x_{1}\mid_{\mathbb{S}_{1}}=U\mid_{\mathbb{S}_{1}}$ and $x_{2}\mid_{\mathbb{S}_{2}}=U\mid_{\mathbb{S}_{2}}$.
Clearly, $x_{1}\mid_{ \mathbb{Z}_{a m \times bn}((i,j))}$ and $x_{2}\mid_{ \mathbb{Z}_{a m \times bn}((i+am,j+bn))}$ are in $\Sigma_{am\times bn}(\mathcal{U})$. By (\ref{eqn:2.2}), it can be verified that the forbidden pattern $U$ on $\mathbb{S}$ occurs in a global pattern of $\mathcal{U}$. This leads a contradiction. Then, $\mathcal{F}=\emptyset$. Therefore, $\mathcal{U}$ is a full shift space. The proof is complete.

\end{proof}

\begin{remark}
\label{remark:2.4}
Proposition \ref{proposition:1.1} can be generalized as follows. For any additive shift space $\mathcal{U}$, if there exists any tessellation $\mathbb{T}$ with

\begin{equation}\label{eqn:2.5}
h_{r}(\mathcal{U})=\frac{1}{|\mathbb{T}|}\log \Gamma_{\mathbb{T}}(\mathcal{U}),
\end{equation}
then $\mathcal{U}$ is a full shift.

\end{remark}

\section{Equal entropies}

\hspace{0.5cm} This section considers the case of $h_{\Omega}(\mathcal{U})=h_{r}(\mathcal{U})$.

Suppose $\Omega=\left\{\Omega(n)\right\}_{n=1}^{\infty}$ is a sequence of finite lattices of $\mathbb{Z}^{2}$ such that $\Omega(n)\subset \Omega(n+1)$
and $\underset{n=1}{\overset{\infty}{\bigcup}} \Omega(n)=\mathbb{Z}^{2}$. For a finite sublattice $\mathbb{L}\subset \mathbb{Z}^{2}$, the interior $\overset{\circ}{\mathbb{L}}$ of $\mathbb{L}$ is defined by

\begin{equation}\label{eqn:3.1}
\overset{\circ}{\mathbb{L}}=\left\{ (i,j)\in \mathbb{L} \hspace{0.1cm} \mid \hspace{0.1cm} (i+1,j), (i,j+1),(i+1,j+1)\in\mathbb{L}  \right\}
\end{equation}
and the boundary $\partial \mathbb{L}$ of $\mathbb{L}$ is defined by

\begin{equation}\label{eqn:3.2}
\partial \mathbb{L}= \mathbb{L}\setminus \overset{\circ}{\mathbb{L}}.
\end{equation}
Let $\mathbb{L}_{1}\subset \mathbb{L}_{2}\subset \mathbb{Z}^{2}$. Define the complement of $\mathbb{L}_{1}$ in $\mathbb{L}_{2}$ by

\begin{equation*}
\mathbb{L}_{1}'=\mathbb{L}_{2}\setminus \mathbb{L}_{1}.
\end{equation*}

For any $k,l\geq 1$, the two-dimensional lattice $\mathbb{Z}^{2}$ can be decomposed as disjoint $k\times k$ rectangular lattices, that is,

\begin{equation}\label{eqn:3.3}
\mathbb{Z}^{2}=\underset{a,b \in \mathbb{Z}}{\bigcup} \hspace{0.1cm} \mathbb{Z}_{k\times l}((ak,bl));
\end{equation}
see Fig. 3.1.

\begin{equation*}
\psfrag{a}{{\small $k$ }}
\psfrag{b}{{\small $l$} }
\includegraphics[scale=1.0]{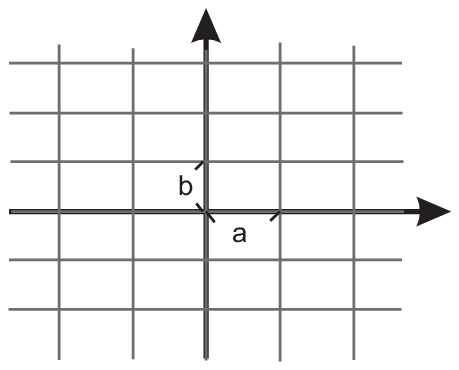}
\end{equation*}
\begin{equation*}
\text{Figure 3.1.}
\end{equation*}

Fix $k,l\geq 1$. For $\Omega=\left\{\Omega(n)\right\}_{n=1}^{\infty}$, let $\mathcal{I}_{k,l}(n)$ be the index set of disjoint $k\times l$ rectangular lattices that are contained in $\Omega(n)$, meaning that

\begin{equation}\label{eqn:3.4}
\mathcal{I}_{k,l}(n)= \left\{ (a,b) \hspace{0.1cm}\mid \hspace{0.1cm}  \mathbb{Z}_{k\times l}((ak,bl))\subseteq  \Omega(n), (a,b)\in\mathbb{Z}^{2}       \right\}.
\end{equation}
Denote by $\alpha_{k,l}(n)$ the cardinal number of $\mathcal{I}_{k,l}(n)$. Let $\Omega_{k,l}(n)$ be the union of all $k\times l$ rectangular lattices in $\mathcal{I}_{k,l}(n)$, so

\begin{equation}\label{eqn:3.4-1}
\Omega_{k,l}(n)= \underset{(a,b)\in\mathcal{I}_{k,l}(n) }{\bigcup} \hspace{0.1cm}  \mathbb{Z}_{k\times l}((ak,bl)).
\end{equation}
Let $\omega_{k,l}(n)$ be the complement of $\Omega_{k,l}(n)$ in $\Omega(n)$, meaning that

\begin{equation}\label{eqn:3.5}
\omega_{k,l}(n)=\Omega(n)\setminus \Omega_{k,l}(n) .
\end{equation}
Denote by $\beta_{k,l}(n)$ the cardinal number of $\omega_{k,l}(n)$. Notably,

\begin{equation}\label{eqn:3.6}
\left|\Omega(n)\right|=\alpha_{k,l}(n) kl+\beta_{k,l}(n).
\end{equation}

The following Lemma specifies the relationship between $\left|\partial\Omega(n)\right|$ and $\beta_{k,l}(n)$.

\begin{lemma}
\label{lemma:3.1}
Let $\Omega=\left\{\Omega(n)\right\}_{n=1}^{\infty}$. Then

\begin{equation}\label{eqn:3.7}
\underset{n\rightarrow\infty}{\limsup} \hspace{0.1cm} \frac{\beta_{k,l}(n)}{\left|\Omega(n)\right|}=0
\end{equation}
for all $k,l\geq 1$ if and only if

\begin{equation}\label{eqn:3.8}
\underset{n\rightarrow\infty}{\limsup} \hspace{0.1cm}\frac{\left|\partial\Omega(n)\right|}{\left|\Omega(n)\right|}=0.
\end{equation}

\end{lemma}

\begin{proof}
(i) ($\Rightarrow$) Firstly, we have

\begin{equation*}
  \partial \Omega (n) \subseteq \partial \Omega_{k,l}(n)\cup \partial \omega_{k,l}(n)\subseteq  \left(  \underset{(a,b)\in\mathcal{I}_{k,l}(n)}{\bigcup} \partial\mathbb{Z}_{k\times l}((ak,bl))\right)\cup \omega_{k,l}(n).
\end{equation*}
Clearly,

\begin{equation*}
\frac{\left| \partial\mathbb{Z}_{k\times l}((ak,bl))   \right|}{\left| \mathbb{Z}_{k\times l}((ak,bl))  \right|}\leq \frac{2(k+l)}{kl}
\end{equation*}
for $k,l\geq 1$. Then, by (\ref{eqn:3.6}) and (\ref{eqn:3.7}),
\begin{equation*}
\begin{array}{rl}
\underset{n\rightarrow\infty}{\limsup} \hspace{0.1cm}\frac{\left|\partial\Omega(n)\right|}{\left|\Omega(n)\right|} \leq &
 \underset{n\rightarrow\infty}{\limsup} \hspace{0.1cm} \frac{1}{\alpha_{k,l}(n)kl+\beta_{k,l}}\left\{\alpha_{k,l}(n)\left(\frac{2(k+l)}{kl}\right)+\beta_{k,l}(n)\right\} \\
 & \\
 \leq &   \underset{n\rightarrow\infty}{\limsup} \hspace{0.1cm} \frac{1}{\alpha_{k,l}(n)kl}\left\{\alpha_{k,l}(n)\left(\frac{2(k+l)}{kl}\right)+\beta_{k,l}(n)\right\} \\
  & \\
 = & \frac{2(k+l)}{(kl)^{2}}
\end{array}
\end{equation*}
for all $k,l\geq 1$. Therefore, (\ref{eqn:3.8}) follows.

(ii) ($\Leftarrow$) Clearly, $\omega_{k,l}(n)$ cannot contain any $\mathbb{Z}_{k\times l}((ak\times bl))$ for any $a,b\in\mathbb{Z}$. If $\omega_{k,l}(n)\cap \mathbb{Z}_{k\times l}((ak\times bl))\neq\emptyset$, then $\partial\Omega(n)\cap \mathbb{Z}_{k\times l}((ak\times bl))\neq\emptyset $. Hence, that

\begin{equation*}
\beta_{k,l}(n)\leq \left|\partial\Omega(n)  \right|(kl-1)
\end{equation*}
can be obtained. Therefore,

\begin{equation*}
\begin{array}{rl}
\underset{n\rightarrow\infty}{\limsup} \hspace{0.1cm}\frac{\beta_{k,l}(n)}{\left|\Omega(n)\right|} \leq &
(kl-1) \hspace{0.1cm} \underset{n\rightarrow\infty}{\limsup} \hspace{0.1cm} \frac{\left|\partial\Omega(n)\right|}{\left|\Omega(n)\right|} \\
 & \\
 = &  0
\end{array}
\end{equation*}
for all $kl>1$. When $k=l=1$, it is clear that $\beta_{1,1}(n)=0$ for all $n\geq 1$. The proof is complete.
\end{proof}

For $\Omega=\left\{\Omega(n)\right\}_{n=1}^{\infty}$, let $\left\{\mathbb{T}(n)\right\}_{n=1}^{\infty}$ be a sequence of tessellations $\mathbb{T}(n)$ such that $\Omega(n) \subset\mathbb{T}(n)$. Denote by $\Omega'(n)=\Omega'(n,\mathbb{T}(n))$ the complement of $\Omega(n)$ in $\mathbb{T}(n)$.

The following theorem is a generalization of Theorem 1.1.

\begin{theorem}
\label{theorem:3.2}
Suppose $\mathcal{U}$ is an additive shift space. Let $\Omega=\left\{\Omega(n)\right\}_{n=1}^{\infty}$. If

\begin{equation}\label{eqn:3.9}
\underset{n\rightarrow\infty}{\limsup} \hspace{0.1cm}\frac{\left|\partial\Omega(n)\right|}{\left|\Omega(n)\right|}=0,
\end{equation}
then
\begin{equation}\label{eqn:3.9-0}
h_{\Omega}(\mathcal{U})\leq h_{r}(\mathcal{U}).
\end{equation}
Furthermore, let $\left\{\mathbb{T}(n)\right\}_{n=1}^{\infty}$ be a sequence of tessellations $\mathbb{T}(n)$ such that $\Omega(n) \subset\mathbb{T}(n)$. If (\ref{eqn:3.9}),
\begin{equation}\label{eqn:3.9-00}
\underset{n\rightarrow\infty}{\limsup} \hspace{0.1cm}\frac{\left|\partial\Omega'(n)\right|}{\left|\Omega'(n)\right|}=0
\end{equation}
and
\begin{equation}\label{eqn:3.9-1}
\sup\left\{\frac{\left|\Omega'(n)\right|}{\left|\Omega(n)\right|}: n\geq 1\right\}<\infty
\end{equation}
holds, then

\begin{equation}\label{eqn:3.10}
h_{\Omega}(\mathcal{U})=h_{r}(\mathcal{U}).
\end{equation}
\end{theorem}

\begin{proof}
(i) ($h_{\Omega}(\mathcal{U})\leq h_{r}(\mathcal{U})$) From (\ref{eqn:3.9}), by Lemma \ref{lemma:3.1}, we have

\begin{equation*}
\underset{n\rightarrow\infty}{\limsup} \hspace{0.1cm} \frac{\beta_{k,l}(n)}{\left|\Omega(n)\right|}=0.
\end{equation*}
From (\ref{eqn:3.6}), it is clear that
\begin{equation*}
\underset{n\rightarrow\infty}{\limsup} \hspace{0.1cm} \frac{\alpha_{k,l}(n)kl}{\left|\Omega(n)\right|}=1.
\end{equation*}
Then, by the sub-additive property,

\begin{equation*}
\begin{array}{rl}
 &\underset{n\rightarrow\infty}{\limsup} \hspace{0.1cm} \frac{1}{\left|\Omega(n)\right|}\log\Gamma_{\Omega(n)}(\mathcal{U})  \\ & \\ \leq &
\left(\underset{n\rightarrow\infty}{\limsup} \hspace{0.1cm} \frac{\alpha_{k,l}(n)kl}{\left|\Omega(n)\right|}\right)\left(\frac{1}{kl}\log\Gamma_{k\times l}(\mathcal{U})\right)+ \left(\underset{n\rightarrow\infty}{\limsup} \hspace{0.1cm} \frac{\beta_{k,l}(n)}{\left|\Omega(n)\right|}\right)\left(\frac{1}{\beta_{k,l}(n)}\log\Gamma_{\omega_{k,l}(n)}(\mathcal{U})\right) \\ & \\
= & \frac{1}{kl}\log\Gamma_{k\times l}(\mathcal{U})
\end{array}
\end{equation*}
for all $k,l\geq 1$. Therefore, $h_{\Omega}(\mathcal{U})\leq \underset{k,l\geq 1}{\inf}\frac{1}{kl}\log\Gamma_{k\times l}(\mathcal{U})=h_{r}(\mathcal{U})$.

(ii) ($h_{\Omega}(\mathcal{U})\geq h_{r}(\mathcal{U})$) 
%
%
Since $\{\mathbb{T}(n)\}_{n=1}^{\infty}$ is a sequence of tesselations, by (\ref{eqn:2.5}),

\begin{equation*}
 \frac{1}{\left|\mathbb{T}(n)\right|}\log\Gamma_{\mathbb{T}(n)}(\mathcal{U})\geq h_{r}(\mathcal{U})
\end{equation*}
for all $n\geq 1$. The sub-additive property implies

\begin{equation*}
\log\Gamma_{\Omega(n)}\geq \log\Gamma_{\mathbb{T}(n)}-\log \Gamma_{\Omega'(n)}.
\end{equation*}
By (\ref{eqn:3.9}), as the proof (i),

\begin{equation*}
h_{\Omega'}(\mathcal{U})\leq h_{r}(\mathcal{U}).
\end{equation*}
Therefore,

\begin{equation*}
\begin{array}{rl}
\underset{n\rightarrow\infty}{\limsup} \hspace{0.1cm} \frac{1}{\left|\Omega(n)\right|}\log\Gamma_{\Omega(n)}(\mathcal{U}) \geq & \underset{n\rightarrow\infty}{\limsup} \frac{|\mathbb{T}(n)|}{|\Omega(n)|}\left(\frac{1}{|\mathbb{T}(n)|}\log\Gamma_{\mathbb{T}(n)}\right)- \underset{n\rightarrow\infty}{\limsup} \frac{|\Omega'(n)|}{|\Omega(n)|}\left(\frac{1}{|\Omega'(n)|}\log\Gamma_{\Omega'(n)}\right) \\
& \\
\geq & \underset{n\rightarrow\infty}{\limsup} \frac{|\mathbb{T}(n)|}{|\Omega(n)|}h_{r}(\mathcal{U})-\underset{n\rightarrow\infty}{\limsup} \frac{|\Omega'(n)|}{|\Omega(n)|}h_{r}(\mathcal{U})\\
& \\
= & h_{r}(\mathcal{U}).
\end{array}
\end{equation*}
The proof is complete.

\end{proof}

The following example illustrates the applications of tessellations in Theorem \ref{theorem:3.2}.

\begin{example}
\label{example:3.1-1}
For $n\geq 1$, let $\Omega=\{\Omega(n)\}_{n=1}^{\infty}$ where

\begin{equation*}
\psfrag{a}{{\footnotesize$n^{2}$}}
\psfrag{b}{{\footnotesize $n$}}
\psfrag{c}{\hspace{-1.0cm} $\Omega(n)=$}
 \includegraphics[scale=1.0]{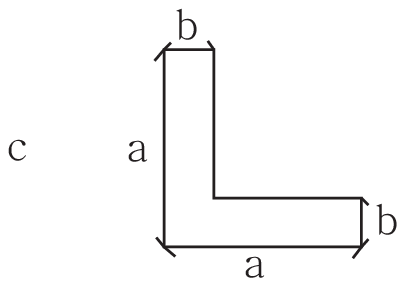}
\end{equation*}
\begin{equation*}
\text{Figure 3.2.}
\end{equation*}
for $n\geq 1$. Clearly, the smallest rectangular lattice that contains $\Omega(n)$ is $\mathbb{Z}_{n^{2}\times n^{2}}$. The size of $\Omega(n)$ is $2n^{3}-n^{2}$, and the size of the complement $\Omega'(n)$ of $\Omega(n)$ in $\mathbb{Z}_{n^{2}\times n^{2}}$ is $\left(n^{2}-n\right)^{2}$. Accordingly,
\begin{equation*}
\sup\left\{\frac{\left|\Omega'(n)\right|}{\left|\Omega(n)\right|}: n\geq 1\right\}=\infty.
\end{equation*}
Hence, Theorem \ref{Theorem:1.1} cannot be utilized to obtain $h_{\Omega}(\mathcal{U})=h_{r}(\mathcal{U})$. However, $\Omega(n)$ is a tessellation for all $n\geq1$. Therefore, by Theorem \ref{theorem:3.2}, $h_{\Omega}(\mathcal{U})=h_{r}(\mathcal{U})$ can be verified.

\end{example}

In the following, we show that when $\mathcal{U}$ is a block gluing shift space, then (\ref{eqn:1.14}) implies (\ref{eqn:1.16}).

\begin{theorem}
\label{theorem:3.3}
Suppose an additive shift space $\mathcal{U}$ is block gluing. If

\begin{equation*}
\underset{n\rightarrow\infty}{\limsup} \hspace{0.1cm}\frac{\left|\partial\Omega(n)\right|}{\left|\Omega(n)\right|}=0,
\end{equation*}
then

\begin{equation*}
h_{\Omega}(\mathcal{U})=h_{r}(\mathcal{U}).
\end{equation*}
\end{theorem}

\begin{proof}
The proof of $h_{\Omega}(\mathcal{U})\leq h_{r}(\mathcal{U})$ is the same as that of Theorem 3.2 (i). Only
$h_{\Omega}(\mathcal{U})\geq h_{r}(\mathcal{U})$ has to be verified. Suppose $\mathcal{U}$ is block gluing with gap $M\geq 1$. Then, $\mathbb{Z}^{2}$ can be arranged as in Fig. 3.3. Here the size of the dashed rectangles is $(k-M)\times (l-M)$.

\begin{equation*}
\psfrag{a}{{\footnotesize $k$ }}
\psfrag{b}{{\footnotesize $l$} }
\psfrag{c}{{\tiny $M$} }
\includegraphics[scale=1.5]{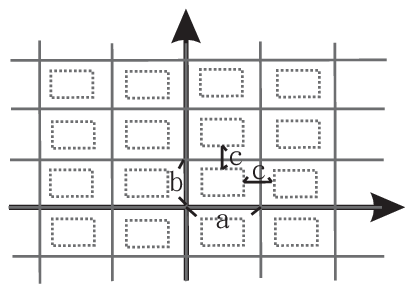}
\end{equation*}
\begin{equation*}
\text{Figure 3.3.}
\end{equation*}
Then, it can be proven that

\begin{equation*}
\Gamma_{\Omega(n)}\geq \Gamma_{\Omega_{k,l}(n)}\geq  \left(\Gamma_{(k-M)\times(l-M)}\right)^{\alpha_{k,l}(n)}
\end{equation*}
for all $k,l\geq M+1$. Hence,

\begin{equation*}
\begin{array}{rl}
\underset{n\rightarrow\infty}{\limsup} \hspace{0.1cm} \frac{1}{\left|\Omega(n)\right|}\log\Gamma_{\Omega(n)}(\mathcal{U})\geq& \left(\underset{n\rightarrow\infty}{\limsup} \frac{\alpha_{k,l}(n)(k-M)(l-M)}{|\Omega(n)|}\right)\left( \frac{1}{(k-M)(l-M)}\log \Gamma_{(k-M)\times (l-M)} \right)\\ & \\ = & \left(\frac{(k-M)(l-M)}{kl}\right)\left(\frac{1}{(k-M)(l-M)}\log \Gamma_{(k-M)\times (l-M)} \right)\\ & \\
\geq & \left(\frac{(k-M)(l-M)}{kl}\right) h_{r}(\mathcal{U})
\end{array}
\end{equation*}
for all $k,l\geq M+1$, where $\Omega_{k,l}(n)$ is defined in (\ref{eqn:3.4-1}). Therefore, $h_{\Omega}(\mathcal{U})\geq h_{r}(\mathcal{U})$. The proof is complete.

\end{proof}

Theorem \ref{theorem:3.2} can be immediately generalized as follows. For brevity, the proof is omitted.

\begin{corollary}
\label{corollary:3.4}
Suppose an additive shift space $\mathcal{U}$ is horizontally (or vertically) block gluing. If $\Omega=\left\{\Omega(n)\right\}_{n=1}^{\infty}$ is
horizontally (or vertically) decomposable and

\begin{equation*}
\underset{n\rightarrow\infty}{\limsup} \hspace{0.1cm}\frac{\left|\partial\Omega(n)\right|}{\left|\Omega(n)\right|}=0,
\end{equation*}
then

\begin{equation*}
h_{\Omega}(\mathcal{U})=h_{r}(\mathcal{U}).
\end{equation*}
\end{corollary}

Notably, $\Omega=\{\Omega(n)\}_{n=2}^{\infty}$ in Example 2.1 is horizontally decomposable, but a sequence $\{\mathbb{T}(n)\}_{n=1}^{\infty}$ of tessellations such that $\Omega$ satisfies (\ref{eqn:3.9-1}) in Theorem 3.2, may not exist.

\section{Unequal entropies}

\hspace{0.5cm} This section concerns the case of $h_{\Omega}(\mathcal{U})>h_{r}(\mathcal{U})$ where $\Omega$ is a generalization for $\Omega_{q}$, and proves Theorem 1.3.

Given a finite lattice $\mathbb{L}\subset\mathbb{Z}^{2}$, for $m\geq 1$, a point $(i,j)\in \mathbb{L}$ has horizontal length $m$ in $\mathbb{L}$ if
$m$ is the largest positive integer such that there exists $(r,s)\in\mathbb{L}$ such that such that $(i,j)\in\mathbb{Z}_{m \times 1}((r,s))\subset \mathbb{L}$. Similarly, a point $(i,j)\in \mathbb{L}$ has vertical length $m$ in $\mathbb{L}$ if $m$ is the largest positive integer such that
there exists $(r',s')\in\mathbb{L}$ such that such that $(i,j)\in\mathbb{Z}_{1 \times m}((r',s'))\subset \mathbb{L}$.

Let $\Omega=\left\{\Omega(n)\right\}_{n=1}^{\infty}$.
For $m\geq 1$, define the subset of $\Omega(n)$ with horizontal and vertical length $m$ by

\begin{equation}\label{eqn:4.1}
\Omega_{m}^{(h)}(n)=\left\{(i,j) \in\Omega(n) \hspace{0.1cm}\mid \hspace{0.1cm}  (i,j) \text{ has horizontal length }m \text{ in }\Omega(n) \right\}
\end{equation}

\begin{equation}\label{eqn:4.2}
\Omega_{m}^{(v)}(n)=\left\{(i,j) \in\Omega(n) \hspace{0.1cm}\mid \hspace{0.1cm}  (i,j) \text{ has vertical length }m \text{ in }\Omega(n) \right\},
\end{equation}
respectively. Denote by $\beta_{m}^{(h)}(n)=\left|\Omega_{m}^{(h)}(n) \right|$ and $\beta_{m}^{(v)}(n)=\left|\Omega_{m}^{(v)}(n) \right|$.

\begin{theorem}
\label{theorem:4.1}
Let $\Omega=\left\{\Omega(n)\right\}_{n=1}^{\infty}$. If there exists $m\geq 1$ such that

\begin{equation}\label{eqn:4.3}
\begin{array}{ccc}
\underset{n\rightarrow\infty}{\limsup} \hspace{0.1cm}\frac{\beta_{m}^{(h)}(n)}{\left|\Omega(n)\right|}>0 & \text{or} &
\underset{n\rightarrow\infty}{\limsup} \hspace{0.1cm}\frac{\beta_{m}^{(v)}(n)}{\left|\Omega(n)\right|}>0,
\end{array}
\end{equation}
then there exists a shift of finite type $\mathcal{U}$ with $h_{\Omega}(\mathcal{U})>h_{r}(\mathcal{U})$.

\end{theorem}
\begin{proof}
Firstly, the case of $\underset{n\rightarrow\infty}{\limsup} \hspace{0.1cm}\frac{\beta_{m}^{(h)}(n)}{\left|\Omega(n)\right|}>0$ is considered. Let $\mathcal{U}=\mathcal{U}_{\mathcal{B}}$ that is defined by (\ref{eqn:1.8}). In the following, we will claim $h_{\Omega}(\mathcal{U}_{\mathcal{B}})>h_{r}(\mathcal{U}_{B})=\log g$.

From the rule of $\mathcal{U}_{\mathcal{B}}$, it can be verified that for any $\mathbb{L}\subseteq \mathbb{Z}^{2}$
\begin{equation*}
\frac{1}{|\mathbb{L}|}\log \Gamma_{\mathbb{L}}(\mathcal{U}_{B})\geq \log g.
\end{equation*}
In particular,
\begin{equation*}
\frac{1}{m}\log \Gamma_{m\times 1}(\mathcal{U}_{B}) > \log g.
\end{equation*}
Then,

\begin{equation*}
\begin{array}{rl}
\underset{n\rightarrow\infty}{\limsup} \hspace{0.1cm}\frac{1}{\left|\Omega(n)\right|}\log\Gamma_{\Omega(n)}(\mathcal{U}_{\mathcal{B}})\geq  & \log g +
\left(\underset{n\rightarrow\infty}{\limsup} \hspace{0.1cm}\frac{\beta_{m}^{(h)}(n)}{\left|\Omega(n)\right|}\right)\left(\frac{1}{m}\log\Gamma_{m\times 1}(\mathcal{U}_{\mathcal{B}})\right)  \\
& \\
> & \log g = h_{r}(\mathcal{U}_{B}).
\end{array}
\end{equation*}
Therefore, the result follows.

For considering the case of $\underset{n\rightarrow\infty}{\limsup} \hspace{0.1cm}\frac{\beta_{m}^{(v)}(n)}{\left|\Omega(n)\right|}>0$, the
shift of finite type $\mathcal{U}_{\mathcal{B}'}\subset \{0,1\}^{\mathbb{Z}^{2}}$ is considered, where the forbidden set $\mathcal{F}'$ of $\mathcal{U}_{\mathcal{B}'}$ is $\mathcal{F}'=\left\{\begin{array}{c} \includegraphics[scale=0.6]{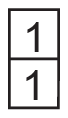} \end{array}\right\}$ , meaning that the basic set of admissible patterns $\mathcal{B}'\subset \{0,1\}^{\mathbb{Z}_{2\times 2}}$ is given as

\begin{equation}\label{eqn:1.8}
\mathcal{B}'=\left\{\begin{array}{ccccccccc}
 \includegraphics[scale=0.6]{0000.eps}, &
  \includegraphics[scale=0.6]{0001.eps}, &
   \includegraphics[scale=0.6]{0010.eps}, &
    \includegraphics[scale=0.6]{0100.eps}, &
     \includegraphics[scale=0.6]{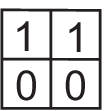}, &
      \includegraphics[scale=0.6]{0110.eps}, &
       \includegraphics[scale=0.6]{1000.eps}, &
        \includegraphics[scale=0.6]{1001.eps}, &
         \includegraphics[scale=0.6]{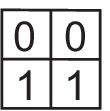}
\end{array}\right\}.
\end{equation}
$\mathcal{U}_{\mathcal{B}'}$ is considered as satisfying Golden-Mean condition $x_{i,j}x_{i,j+1}=0$ in the vertical direction and no constraint in the horizontal direction.

By a similar argument, it can be proven that  $h_{\Omega}(\mathcal{U}_{\mathcal{B}'})>h_{r}(\mathcal{U}_{B}')=\log g$. The proof is complete.

\end{proof}

\begin{example}
\label{example:4.1}
$\Omega_{q}=\left\{\Omega_{q}(n)\right\}_{n=1}^{\infty}$, $q\geq 2$, satisfies $\underset{n\rightarrow\infty}{\limsup} \hspace{0.1cm}\frac{\beta_{2}^{(h)}(n)}{\left|\Omega(n)\right|}>0$. By Theorem \ref{theorem:4.1}, the result (\ref{eqn:1.13}) can be recovered.
\end{example}

The spatial entropy $h_{\Omega}(\mathcal{U})$, such that $\Omega=\{\Omega(n)\}_{n=1}^{\infty}$ contains a lower-dimensional part whose size is comparable to that of its two-dimensional part, is closely related to the projectional entropy \cite{30,31}. Johnson \emph{et al.} introduced the projectional entropy of a $d$-dimensional, $d\geq 2$, shift space $\mathcal{U}$ as follows \cite{30}. Let $\mathcal{V}=\{\vec{\mathbf{v}}_{j}\in\mathbb{Z}^{d}  \hspace{0.1cm} \mid \hspace{0.1cm} 1\leq j\leq r\}$, $1\leq r< d$, be linear independent integral vectors in $\mathbb{Z}^{d}$. Let

\begin{equation*}
\mathbb{L}=\left\{s_{1}\vec{\mathbf{v}}_{1}+s_{2}\vec{\mathbf{v}}_{2}+\cdots+s_{r}\vec{\mathbf{v}}_{r} \hspace{0.1cm} \mid \hspace{0.1cm} s_{j}\in\mathbb{Z}, 1\leq j\leq r \right\}\subset\mathbb{Z}^{d}
\end{equation*}
be the subspace of $\mathbb{Z}^{d}$ spanned by integer multiples of vectors in $\mathcal{V}$. Then, the $\mathbb{L}$ projectional entropy $h_{\mathbb{L}}(\mathcal{U})$ of $\mathcal{U}$ is the topological entropy of the $\mathbb{Z}^{r}$-shift $\mathcal{U}\mid_{\mathbb{L}}$ with the $\mathbb{Z}^{r}$ shift action $\sigma\mid_{\mathbb{L}\times \mathcal{U}\mid_{\mathbb{L}}}$, so

\begin{equation*}
h_{\mathbb{L}}(\mathcal{U})=h_{top}(\mathcal{U}\mid_{\mathbb{L}}).
\end{equation*}

Johnson \emph{et al.} showed the topological entropy $h_{top}(\mathcal{U})$ is the lower bound of all projectional entropy $h_{\mathbb{L}}(\mathcal{U})$. Notably, the topological entropy $h_{top}(\mathcal{U})$ equals the rectangular spatial entropy $h_{r}(\mathcal{U})$, so

\begin{equation*}
h_{\mathbb{L}}(\mathcal{U})\geq h_{r}(\mathcal{U})
\end{equation*}
for all subspaces $\mathbb{L}\subset\mathbb{Z}^{d}$.

Moreover, Johnson \emph{et al.} also proved that if $\mathcal{U}$ is an extendable and block gluing $\mathbb{Z}^{2}$ shift of finite type and $\mathbb{L}\subset\mathbb{Z}^{2}$ is a one-dimensional sublattice, then $h_{top}(\mathcal{U})=h_{\mathbb{L}}(\mathcal{U})$ if and only if $\mathcal{U}=(\mathcal{U}\mid_{\mathbb{L}})^{\mathbb{Z}}$. Here, an additive shift space $\mathcal{U}$ is called extendable if for any allowable rectangular pattern $U_{m\times n}$ on $\mathbb{Z}_{m\times n}$, $U_{m\times n}$ can be extended to be a global pattern $U\in\mathcal{U}$ on $\mathbb{Z}^{2}$.

Schraudner \cite{31} constructed a $\mathbb{Z}^{3}$ shift of finite type, called  the electrical wire shift, and proved that the result of Johnson \emph{et al.} \cite{30} cannot be generalized to a $\mathbb{Z}^{d}$ shift of finite type, $d\geq 3$, under extendable and block gluing conditions. Schraudner also used a stronger mixing property, the uniform filling property, (UFP)
to prove that if $\mathcal{U}$ is a $\mathbb{Z}^{d}$ shift, $d\geq 2$, and $\mathbb{L}\subset\mathbb{Z}^{d}$ is a $r$-dimensional sublattices, $1\leq r < d$, then   $h_{top}(\mathcal{U})=h_{\mathbb{L}}(\mathcal{U})$ if and only if $\mathcal{U}=(\mathcal{U}\mid_{\mathbb{L}})^{\mathbb{Z}^{d-r}}$.

Now, for any $\mathbb{Z}^{2}$ shift space $\mathcal{U}$, let

\begin{equation}\label{eqn:4.4}
\hat{h}^{(1)}(\mathcal{U})=\sup \left\{h_{\mathbb{L}}(\mathcal{U})  \hspace{0.1cm} \mid \hspace{0.1cm} \mathbb{L} \text{ is an one-dimensional sublattice} \right\}.
\end{equation}
Clearly,
\begin{equation*}
\hat{h}^{(1)}(\mathcal{U})\geq h_{r}(\mathcal{U}).
\end{equation*}

Therefore, the following theorem can be obtained.

\begin{theorem}
\label{theorem:4.2}
Let $\mathcal{U}$ be a block gluing $\mathbb{Z}^{2}$ shift space. If

\begin{equation}\label{eqn:4.4-1}
h_{r}(\mathcal{U}) < \hat{h}^{(1)}(\mathcal{U}),
\end{equation}
then for any $h\in\left[h_{r}(\mathcal{U}),\hat{h}^{(1)}(\mathcal{U})\right)$, there exists $\Omega=\{\Omega(n)\}_{n=1}^{\infty}$ such that

\begin{equation}\label{eqn:4.5}
h_{\Omega}(\mathcal{U})=h.
\end{equation}
Furthermore, if $\hat{h}^{(1)}$ can be attained by some one-dimensional sublattice $\mathbb{L}'$, then there exists $\Omega=\{\Omega(n)\}_{n=1}^{\infty}$ such that
\begin{equation}\label{eqn:4.5}
h_{\Omega}(\mathcal{U})=\hat{h}^{(1)}=h_{\mathbb{L}'}(\mathcal{U}).
\end{equation}
\end{theorem}

\begin{proof}

Suppose $h\in\left[h_{r}(\mathcal{U}),\hat{h}^{(1)}(\mathcal{U})\right)$. It is clear that there exists a one-dimensional sublattice $\mathbb{L}$ with $h_{\mathbb{L}}(\mathcal{U})>h$.
Let $\mathbb{L}=\{s\vec{\mathbf{v}} \hspace{0.1cm}  \mid \hspace{0.1cm} s\in\mathbb{Z}\}$ for some $\vec{\mathbf{v}}\in\mathbb{Z}^{2}$.
 Then, there exists $0<a\leq 1$ such that
$h=a h_{r}(\mathcal{U})+(1-a) h_{\mathbb{L}}(\mathcal{U})$.

Now, let $\Omega=\{\Omega(n)\}_{n=1}^{\infty}$ where $\Omega(n)$ is the union of the two-dimensional part $\mathbb{Z}_{n\times n}$ and the one-dimensional part
 \begin{equation*}
 \mathbb{L}(n)=\{s\vec{\mathbf{v}}+(n,0) \hspace{0.1cm}  \mid \hspace{0.1cm}  0\leq s\leq b(n)\}
 \end{equation*}
 with
\begin{equation*}
\underset{n\rightarrow\infty}{\lim} \hspace{0.2cm}\frac{n^{2}}{n^{2}+b(n)+1}=a;
\end{equation*}
see Fig. 4.1.

\begin{equation*}
\begin{array}{c}
\psfrag{d}{$\Omega(n)=$}
\psfrag{a}{$n$}
\psfrag{b}{$\vec{\mathbf{v}}$}
\psfrag{c}{{\footnotesize \hspace{-0.5cm} $\mathbb{L}(n)$}}
\includegraphics[scale=1.0]{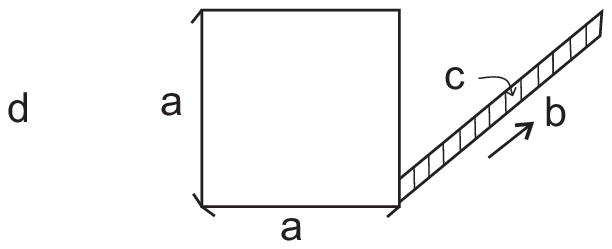}
\end{array}
\end{equation*}
\begin{equation*}
\text{Figure 4.1.}
\end{equation*}
Since $\mathcal{U}$ is block gluing, it can be easily verified that

\begin{equation*}
h_{\Omega}(\mathcal{U})= a h_{r}(\mathcal{U})+(1-a) h_{\mathbb{L}}(\mathcal{U})=h.
\end{equation*}
The proof is complete.

\end{proof}


\begin{thebibliography}{10}

\bibitem{0} P. Ballister, B. Bollob$\acute{a}$s, and A. Quas, \textit{Entropy Along Convex Shapes, Random Tilings and Shifts of Finite Type}, Illinois journal ofMatlaematics \textbf{46}, (2002) 781--795.



\bibitem{0-1} J.C. Ban, W.G. Hu and S.S. Lin, \textit{Pattern generation problems arising in multiplicative integer systems}, submitted. (arXiv:1207.7154)



\bibitem{0-1-1} J.C. Ban, W.G. Hu, S.S. Lin and Y.H. Lin,  \textit{Zeta functions for two-dimensional shifts of finite type}, Memo. Amer. Math. Soc. \textbf{221}, no. 1037 (2013).


\bibitem{0-2} J.C. Ban, W.G. Hu, S.S. Lin and Y.H. Lin, \textit{Association of mosaic patterns in two-dimensional lattice dynamical systems}, submitted.


\bibitem{1} J.C. Ban and S.S. Lin, \textit{Patterns generation and
transition matrices in multi-dimensional lattice models}, Discrete
Contin. Dyn. Syst. \textbf{13}, no. 3, (2005) 637--658.

\bibitem{2} J.C. Ban, S.S. Lin and Y.H. Lin,  \textit{Patterns generation and
spatial entropy in two dimensional lattice models}, Asian J. Math. \textbf{11}, no. 3, (2007) 497--534.



\bibitem{11} M. Boyle, R. Pavlov and M. Schraudner, \textit{Multidimensional sofic shifts without separation and their factors}, Trans. Amer. Math. Soc. \textbf{362}, (2010) 4617--4653.


\bibitem{12} S. N. Chow, J. Mallet-Paret, and E. S. Van Vleck, \textit{Pattern formation and spatial chaos
in spatially discrete evolution equations}, Random Comput. Dynam. \textbf{4}, (1996) 109--178.

\bibitem{17} A.H. Fan, L.M. Liao and J.H. Ma, \textit{Level sets of multiple ergodic averages}, Monatsh. Math. \textbf{168} (2012) 17--26.

\bibitem{18} A.H. Fan, J. Schmeling and M. Wu, \textit{Multifractal analysis of multiple ergodic averages}, C. R. Math.
Acad. Sci. Paris \textbf{349}, (2011) 961–-964.


\bibitem{25} M. Hochman and T. Meyerovitch, \textit{A characterization of the entropies of multidimensional shifts of finite type}, Annals of Mathematics \textbf{171}, (2010) 2011--2038.

\bibitem{26} B. Host and B. Kra, \textit{Nonconventional ergodic averages and nilmanifolds}, Annals of Math. \textbf{161}, (2005) 397–-488.


\bibitem{29} W.G. Hu and S.S. Lin, \textit{Nonemptiness problems of plane square tiling with two colors}, Proc. Amer. Math. Soc. \textbf{139}, (2011) 1045--1059.


\bibitem{30} A. Johnson, S. Kass and K. Madden, \textit{Projectional entropy in higher dimensional shifts of finite type}, Complex Systems \textbf{17} no. 3,  (2007) 243–-257.

\bibitem{32} R. Kenyon, Y. Peres, and B. Solomyak, \textit{Hausdorff dimension of the multiplicative golden mean shift}, C. R. Math. Acad. Sci. Paris, \textbf{349}, (2011) 625--628.




\bibitem{33} R. Kenyon, Y. Peres and B. Solomyak, \textit{Hausdorff dimension for fractals invariant under the
multiplicative integers}, Ergodic Theory Dynam. Sys. \textbf{32}, (2012) 1567--1584.



\bibitem{33-1} D. Lind and B. Marcus, \textit{An introduction to symbolic
dynamics and coding}, Cambridge University Press, Cambridge (1995).


\bibitem{33-2} N. G. Markley and M. E. Paul, \textit{Maximal measures and entropy for $Z^{\nu}$ subshift of finite type}, Classical Mechanics and Dynamical Systems, Lecture Notes in Pure and Appl. Math. \textbf{70} (Medford, Mass.), (1979) 135--157.

\bibitem{33-3} N.G. Markley and M.E. Paul, \textit{Matrix subshifts for $Z\sp{\nu }$ symbolic dynamics}, Proc. London Math. Soc. \textbf{43}, (1981) 251-272.


\bibitem{43} Y. Peres, J. Schmeling, S. Seuret and B. Solomyak, \textit{Dimensions of some fractals defined via the semigroup generated by $2$ and $3$}, Israel Journal of Mathematics \textbf{199}, (2014) 687--709.




\bibitem{44} Y. Peres, B. Solomyak, \textit{Dimension spectrum for a nonconventional ergodic average}, Real Analysis Exchange \textbf{37}, (2012) 375--388.

\bibitem{45} Y. Pesin and H. Weiss, \textit{The multifractal analysis of Birkhoff averages and large deviation}, in
"Global Analysis of Dynamical Systems", Inst. Phys., Bristol, (2001) 419–-431.


\bibitem{31} M. Schraudner, \textit{Projectional entropy and the electrical wire shift}, Discrete Contin. Dyn. Syst. \textbf{26} no. 1, (2010) 333–-346.

\end{thebibliography}
\end{document}